\documentclass[12pt,a4paper,twoside, headrule]{amsart}
\usepackage{amsfonts, amsmath, amssymb,latexsym}
\usepackage{epsfig}
\usepackage[curve]{xy}
\usepackage[french,english]{babel}

\footskip=10mm\parskip 0.2cm\addtocounter{page}{0}
\newtheorem{thm}{Theorem}[section]
\newtheorem{cor}[thm]{Corollary}
\newtheorem{lem}[thm]{Lemma}
\newtheorem{prop}[thm]{Proposition}
\theoremstyle{definition}
\newtheorem{defn}[thm]{Definition}

\theoremstyle{remark}
\newtheorem{rem}[thm]{Remark}
\numberwithin{equation}{section}

\title{UNE FORMULE $P$-ADIQUE DU NOMBRE DES CLASSES D'ORDRE SUPERIEUR}

\author[I. Blanco-Chac\'{o}n]{Iv\'{a}n Blanco-Chac\'{o}n}
\thanks{Partially supported by the project MTM2009-07024 of Spanish \textit{Ministerio de Ciencia e Innovaci\'{o}n}.}
\address{Facultad de Matem\'{a}ticas,
Universidad de Barcelona,
Gran Via de les Corts Catalanes, 585.
08007 Barcelona,
Spain.
}

\begin{document}

\renewcommand\baselinestretch{1.2}
\renewcommand{\arraystretch}{1}
\def\base{\baselineskip}
\font\tenhtxt=eufm10 scaled \magstep0 \font\tenBbb=msbm10 scaled
\magstep0 \font\tenrm=cmr10 scaled \magstep0 \font\tenbf=cmb10
scaled \magstep0


\def\evenhead{{\protect\centerline{\textsl{\large{Iv\'{a}n\,Blanco-Chac\'{o}n}}}\hfill}}

\def\oddhead{{\protect\centerline{\textsl{\large{On the Kubota-Leopoldt $p$-adic $L$-function$}}}\hfill}}

\pagestyle{myheadings} \markboth{\evenhead}{\oddhead}

\thispagestyle{empty}

\begin{center}
\textbf{A HIGHER ORDER $P$-ADIC CLASS NUMBER FORMULA}
\end{center}
\maketitle

\begin{abstract}
We generalize a formula of Leopoldt which relates the $p$-adic regulator modulo $p$ of a real abelian extension of $\mathbb{Q}$ with the value of the relative Dedekind zeta function at $s=2-p$. We use this generalization to give a statement on the non-vanishing modulo $p$ of this relative zeta function at the point $s=1$ under a mild condition.
\end{abstract}

\renewcommand{\abstractname}{R\'{e}sum\'{e}}
\begin{abstract}
On g\'{e}n\'{e}ralise une formule de Leopoldt qui relie le r\'{e}gulateur $p$-adique modulo $p$ d'une extension ab\'{e}lienne r\'{e}elle de $\mathbb{Q}$ avec la valeur de la fonction zeta de Dedekind relative en $s=2-p$. On utilise cette g\'{e}n\'{e}ralisation pour donner un \'{e}nonc\'{e} sur la non-annulation modulo $p$ de cette fonction zeta relative au point $s=1$ sous une faible condition.
\end{abstract}

Let $K/\mathbb{Q}$ be a real abelian extension  of degree $g$, class number $h$ and discriminant $d$. Denote by $\zeta_K$ the Dedekind zeta function of $K$ ($\zeta$ is the Riemann zeta function). In what follows, we will suppose that $p>3$ is a prime number. For any integer number $n\geq 1$, $B_{n,\chi}$ is the generalized Bernoulli number attached to $n$ and $\chi$. We fix once and for all an embedding of $\mathbb{Q}^{alg}$ into $\mathbb{Q}_p^{alg}$, where for any field $L$, $L^{alg}$ denotes an algebraic closure of $L$.

\section{The $p$-adic class number formula mod $p$}
Let $\chi$ be a primitive character of $\mathrm{Gal}(K/\mathbb{Q})$ of conductor $f_{\chi}$ and $\tau(\chi)$ its Gauss sum. In \cite{leopoldt}, Leopoldt introduced the following $p$-adic analogue of the complex value $L(1,\chi)$
$$
\mathcal{L}_p(\chi):=-\frac{\tau(\chi)}{f_{\chi}}\sum_{a=1}^{f_{\chi}}\overline{\chi}(a)\log_p\left(1-\xi^a\right),
$$
where $\xi$ is a primitive $f_{\chi}$-root of $1$ and $\log_p$ is the Iwasawa logarithm.
\begin{prop}[Leopoldt, \cite{leopoldt}, 1.8]Suppose that $p\nmid f_{\chi}$ and that $\chi(-1)=1$. Then
$$
\mathcal{L}_p(\chi)\equiv \overline{\chi(p)}\frac{B_{p-1,\chi}}{p-1}p\pmod{p^2}.
$$
\label{leopoldt1}
\end{prop}
The following definition is well known.
\begin{defn}[cf.\,Leopoldt,\,\cite{leopoldt}]For any $ z\in\mathbb{Z}_p^*$, the Fermat quotient of $z\pmod{p}$, denoted by $Q_p(z)$, is the reminder
of $(z^{p-1}-1)/p\pmod{p}$.
\end{defn}
The Fermat quotient satisfies the following property, whose proof is straightforward.
\begin{prop}For any $z\in\mathbb{Z}_p^*$,
$$
\log_p(z)\equiv -pQ_p(z)\pmod{p^2}.
$$
\label{fermatquotient}
\end{prop}
Let $\{\varepsilon_k\}_{k=1}^{g-1}$ be a fundamental system of units of $\mathcal{O}_K$, the ring of integers of $K$, and $\{\sigma_j\}_{j=1}^g$ a set of real embeddings of $K$ into $\mathbb{Q}^{alg}$.
\begin{defn}[Leopoldt, \cite{leopoldt}] The $p$-adic regulator mod $p$ is
$$R^{(p)}(K)=\mbox{det}\left(Q_p\left(\sigma_j\left(\varepsilon_k\right)\right)\right)_{j,k=1}^{g-1}.$$
\end{defn}
As in the case of the standard $p$-adic regulator, which we denote by $R_p(K)$, it is easy to check that $R^{(p)}(K)$ is well defined up to a sign. The following theorem is stated in \cite{leopoldt} but it appears without a proof, which we give here.
\begin{thm}Let $K/\mathbb{Q}$ be a real abelian extension and let $p>3$ be a prime such that for any character $\chi$ of
$\mathrm{Gal}\left(K/\mathbb{Q}\right)$ of conductor $f_{\chi}$, $p\nmid f_{\chi}$. Then, with the above notations, there exists a fundamental system of units such that
$$
\frac{2^{g-1}hR^{(p)}(K)}{\sqrt{d}}\equiv\frac{\zeta_K(2-p)}{\zeta(2-p)}\,\pmod{p}.
$$
\label{classnumberzeta}
\end{thm}
\begin{proof}
From the very definition of $p$-adic regulator mod $p$, for any fundamental system of units we have that $R_p(K)\equiv(-p)^{g-1}R^{(p)}(K)\pmod{p^g}$. From the $p$-adic class number formula, there exists a fundamental system of units such that
$$
\dfrac{2^{g-1}hR_p(K)}{\sqrt{d}}=\prod\mathcal{L}_p(\chi),
$$
the product being taken over the set of non trivial characters of $\mathrm{Gal}(K/\mathbb{Q})$. Using Proposition \ref{leopoldt1}, we obtain that
\begin{equation}
\prod_{\chi\neq 1}\mathcal{L}_p\left(\chi\right)\equiv (-p)^{g-1}\prod_{\chi\neq 1}L\left(2-p;\chi\right)\pmod{p^g}.
\label{congruenciapadic2}
\end{equation}
But the left hand side of \eqref{congruenciapadic2} is congruent to $\dfrac{(-p)^{g-1}2^{g-1}hR^{(p)}(K)}{\sqrt{d}}\pmod{p^g}$.
\end{proof}
\begin{rem}In formula (3.8) of \cite{leopoldt}, the right hand side of the congruence appears with a potentially distinct sign. This is no contradiction, since regulators are defined up to a sign.
\end{rem}
\section{A $p$-adic class number formula mod $p^n$ and special values of $p$-adic relative zeta functions}
Given $z\in\mathbb{Z}_p^*$, write $z=\omega(z)\langle z\rangle$, where $\omega(z)$ is the unique $(p-1)$-th root of $1$ which is congruent to $z$ modulo $p$. Denote by $Q_{p,n}(z)$ the truncation of the series
$\frac{-1}{p}\mathrm{log}_p(1+p\tilde{z})\pmod{p^{n+1}}$.
\begin{defn}Let $\{\varepsilon_k\}_{k=1}^{g-1}$ be a fundamental system of units of $K$ and $\{\sigma_j\}_{j=1}^g$ a set of embeddings of $K$ into $\mathbb{Q}^{alg}$. The $p$-adic regulator modulo $p^n$ is
$$
R^{(p,n)}(K)=\mathrm{det}\left(Q_{p,n}\left(\sigma_j(\varepsilon_k)\right)\right)_{1\leq j,k\leq g-1}.
$$
\end{defn}
It is also easy to prove that $R^{(p,n)}(K)$ is also well defined up to a sign. The following lemma is obvious from the definition of $Q_{p,n}$.
\begin{lem}For any integer $n\geq 1$, $-pQ_{p,n}(z)\equiv\mathrm{log}_p(z)\pmod{p^{n+2}}$.
\label{fermatquotient2}
\end{lem}
As a direct consequence of this lemma, we have the following result.
\begin{prop}For any integer $n\geq 1$, there exists a fundamental system of units such that $R_p(K)\equiv (-p)^{g-1}R^{(p,n)}(K)\pmod{p^{n+g}}$.
\label{fermatquotient3}
\end{prop}$\Box$

Let $L_p(s;\chi)$ be the Kubota-Leopoldt $p$-adic $L$-function attached to $\chi$ (cf. \cite{kubota1}, \cite{iwasawa}).

\begin{prop}[Kubota-Leopoldt,\;\cite{kubota1}] Let $\chi$ be a character of $\mathrm{Gal}(K/\mathbb{Q})$ of conductor $f_{\chi}$. Let $p>3$ be a prime and suppose that $p\nmid f_{\chi}$. Then, for any $s\in\mathbb{Z}_p$,
$$
|L_p(1;\chi)-L_p(1-s;\chi)|_p\leq|ps|_p.
$$
\label{padicstimation}
\end{prop}$\Box$

\begin{thm}[Leopoldt, \cite{kubota2}] $L_p(1;\chi)=\mathcal{L}_p(\chi)$.
\label{interpolationkubota}
\end{thm}
$\Box$

Now, we state and prove our main result.

\begin{thm}Let $K/\mathbb{Q}$ be a real abelian extension and $p>3$ an unramified  prime for $K$ such that for any character $\chi$ of
$\mathrm{Gal}\left(K/\mathbb{Q}\right)$, $p\nmid f_{\chi}$. Then, there exists a fundamental system of units in $\mathcal{O}_K^*$ such that for any $n\geq 1$,
$$
\frac{2^{g-1}hR^{(p,n)}(K)}{\sqrt{d}}\equiv\frac{\zeta_K(1-p^n(p-1))}{\zeta(1-p^n(p-1))}\,\pmod{p^{n+1}}.
$$
\label{classnumberzeta2}
\end{thm}
\begin{proof}Given an integer $n\geq 1$, by using Lemma \ref{fermatquotient2} we have that $R_p(K)\equiv
(-p)^{g-1}R^{(p,n)}(K)\pmod{p^{n+g}}$. By using Proposition \ref{padicstimation} and Theorem \ref{interpolationkubota}, we obtain
\begin{equation}
\prod_{\chi\neq 1}\mathcal{L}_p\left(\chi\right)\equiv p^{g-1}\prod_{\chi\neq
1}L\left(1-p^n(p-1);\chi\right)\pmod{p^{n+g}}.
\label{congruenciapadic1}
\end{equation}
By the $p$-adic class number formula, the left hand side of \eqref{congruenciapadic1} equals $\dfrac{2^{g-1}hR_p(K)}{\sqrt{d}}$, which is congruent to
$\dfrac{(-p)^{g-1}2^{g-1}hR^{(p,n)}(K)}{\sqrt{d}}\pmod{p^{n+g}}$. Dividing by $p^{g-1}$, the result follows.
\end{proof}
With Proposition \ref{classnumberzeta2}, we can formulate the following result on the non-vanishing mod $p$ of the classical Dedekind zeta function
\begin{cor}Let $K/\mathbb{Q}$ be a real abelian extension. Let $p>3$ be an unramified prime such that $(p,h_k)=1$ and $R_p(K)\in
p^{g-1}\mathbb{Z}_p^*$. Then, for any $n\geq 1$,
$$
p\nmid\dfrac{\zeta_K\left(1-p^n(p-1)\right)}{\zeta(1-p^n(1-p))}.
$$
\label{nonvanishingmodp}
\end{cor}
\begin{proof}
If $R_p(K)\in p^{g-1}\mathbb{Z}_p^*$, then, for any $n\geq 1$, $R^{(p,n)}(K)\in\mathbb{Z}_p^*$. Since $p$ is unramified at $K$, $(p,d_K)=1$. Since $(p,h_K)=1$, the left hand side of the equality of Theorem \ref{classnumberzeta2} is a $p$-adic unit. Hence, the result follows.
\end{proof}
\begin{defn}(Iwasawa, \cite{iwasawa})
Let $K/\mathbb{Q}$ be a real abelian extension of $\mathbb{Q}$. The $p$-adic Dedekind zeta function of $K$ is $\zeta_{K,p}(s)=\displaystyle\prod
L_p(1-s;\chi)$, where the product runs through the group of characters of Gal$(K/\mathbb{Q})$.
\end{defn}
The function $\zeta_{K,p}$ is meromorphic with a simple pole at $s=0$, which can be cancelled out dividing by $L_p(1-s;1)$. The quotient is, hence, an analytic function denoted by $\zeta_{K/\mathbb{Q},p}$. By passing to the limit in Corollary \ref{nonvanishingmodp}, we have the
following result.
\begin{thm}Let $K/\mathbb{Q}$ be a real abelian extension. Let $p$ be an unramified prime such that $(p,h_k)=1$ and $R_p(K)\in
p^{g-1}\mathbb{Z}_p^*$. Then, $$
\left|\zeta_{K/\mathbb{Q},p}(1)\right|_p=1.
\label{nonvanishingmodp2}
$$
\end{thm}

\emph{Acknowledegements.} I thank Pilar Bayer for carefully reading this note and for some valuable comments which allowed me to improve the earlier version, and Paloma Bengoechea, who has translated the abstract into French.

\end{document}